\documentclass[preprint,3p,twocolumn]{elsarticle}
\usepackage{amsmath, amsthm, amsfonts, amssymb, graphicx, cite, epsfig,epstopdf,color,soul}
\usepackage{url}

\newtheorem{theorem}{Theorem}
\newtheorem{proposition}[theorem]{Proposition}
\newtheorem{lemma}[theorem]{Lemma}

\theoremstyle{definition}

\theoremstyle{remark}

\def\C{\mathbb{C}}
\def\R{\mathbb{R}}
\def\X{\mathcal{X}}
\def\cap{{\rm cap}({\mathcal{X}})}
\def\Y{\mathcal{Y}}
\def\tm{t_{\rm min}}
\def\E{\mathcal{E}}

\begin{document}
\begin{frontmatter}
\title{\textbf{Eigenvalue clustering, control energy, and logarithmic capacity}}

\author[ise]{Alex Olshevsky\corref{ack}}

\address[ise]{Department of ISE, University of Illinois at Urbana-Champaign, USA \\ {\tt aolshev2@illinois.edu}}
\cortext[ack]{This research was supported by NSF grant ECCS-1351684.}
\begin{abstract} We prove two bounds showing that if the eigenvalues of a matrix are clustered in a region of the complex plane then the corresponding discrete-time linear system requires significant energy to control.  A curious feature of one of our bounds is that the dependence on the region  is via its logarithmic capacity, which is a measure of how well a unit of mass may be spread out over the region to minimize a  logarithmic potential. 
\end{abstract}
\end{frontmatter}

\section{Introduction} We will consider discrete-time linear systems 
\begin{equation} \label{maineq} x(t+1) = A x(t) + B u(t),  \end{equation} where $A \in \C^{n \times n}$ and $B \in \C^{n \times k}$. Our goal is to understand the relation between the locations of the eigenvalues of $A$ within the complex plane and the energy needed to steer Eq. (\ref{maineq}) by choosing the input $u(t)$. We will prove two bounds to the effect that if the eigenvalues of $A$ are  clustered, then Eq. (\ref{maineq}) is ``difficult to control'' in the sense of requiring large inputs to steer between states.

Our work is related to a growing body of literature investigating the control properties of large-scale systems. A strand of this literature, to which this paper belongs, is to identify the fundamental limitations for controlling such networks \citep{cortes, tzoumas1, tzoumas2, summers, pasq, zamp2, pasq2, motter}. A common concern is control difficulty when $n$ (the number of states) is large; it has been experimentally observed that in some scenarios the minimum control energy grows exponentially as a function of $n$ \citep{motter, barabasi}. In this note, we will study how eigenvalue locations of $A$ in the complex plane can sometimes be a de-facto obstacle to efficient control for systems with many states. 

It is textbook material that the energy needed to steer a linear system is related to the smallest eigenvalue of the controllability Gramian, and we spell this out before describing the problem and our results. Given initial state $x_0$ and final state $x_{\rm f}$, we let $\E( A, B, x_0 \rightarrow x_{\rm f}, t) $ be the minimal energy $\sum_{i=0}^{t-1} ||u(i)||_2^2$ among all inputs which result in $x(t)=x_{\rm f}$ starting from $x(0) = x_0$. We then use this notion to define the ``difficulty of controllability'' of a linear system by considering the worst-case energy needed to move from the origin to a point on the unit sphere, i.e., 
\[ \E(A,B, t) := \sup_{||y||_2 = 1} ~\E(A,B, 0 \rightarrow y, t). \] We will allow both sides to be infinite if there is a vector $y$ on the unit sphere which cannot be reached with any choice of $u(0), \ldots, u(t-1)$. This is not the only way to formalize the difficulty of controllability of a linear system (for example, one might also consider the expected energy to move to a random point on the unit sphere) but it is among the most natural. Defining the $t$-step controllability Gramian as \begin{equation} \label{gramian} W(t) := \sum_{i=0}^t A^i  B B^* (A^*)^i, \end{equation} basic linear algebra then gives that
\[ \mathcal{E}(A,B,t) = \frac{1}{\lambda_{\rm min}(W(t-1))}, \] 
where $\lambda_{\rm min}(W(t-1))$ is the smallest eigenvalue of the nonnegative definite matrix $W(t-1)$. 

Thus the question of how difficult a system is to control (in a certain worst-case sense) reduces to the analysis of the smallest eigenvalue of the controllability Gramian. The study of that eigenvalue is the subject of this note. 

We are motivated by a recent result of \citep{pasq} which showed that if $A$ is a diagonalizable matrix with $m$ eigenvalues within the circle $\X = \{ z ~|~ |z| \leq \mu < 1\}$, then the smallest eigenvalue of  $W(t)$ for any $t$ is upper bounded by a product of two terms one of which is $\mu^{2 (\lceil m/k \rceil -1) }/(1-\mu^2)$ (where recall $k$ is the number of columns of $B$). In other words, if $m$ is large enough compared to $k$ and $\mu$ is not close to $1$, then the smallest eigenvalue of $W(t)$ is exponentially close to zero. 

Our goal here is to produce similar results for other sets $\X$, especially those  which are not contained within the interior of the unit circle. In this case we will not be able to obtain bounds on $\lambda_{\rm min}(W(t))$ for all $t$, but we will be able to upper bound this eigenvalue for some concrete choices of time $t$. 

\subsection{Related work} As already mentioned, the main motivating work for the present paper is \citep{pasq}, which was the first (to our knowledge) to obtain results connecting eigenvalue clustering to lower bounds on control energy. Follow-up work included \citep{pasq2}, which studied the relation between the difficulty of controllability and the propagation of inputs by the system in various directions, and \citep{zamp2} which studied connections to measures of centrality such as PageRank. 

The existence of small eigenvalues of the discrete-time controllability Gramian appears to have not attracted significant attention in the existing control literature beyond the above papers. In continuous time, results on the condition number (which is the ratio of the largest and smallest eigenvalue of the Gramian) in the case when $A$ is stable have been derived \citep{penzl}, as well as more general results on ratios of eigenvalues \citep{antoulas, truhar}.

Our work is also related to a series of recent preprints analyzing properties of eigenvalues of the Gramian of linear \citep{tzoumas1, tzoumas2, summers} and bilinear \citep{cortes} systems. These papers studied the efficiency of algorithms for placement of sensors and actuators, as well as the underlying properties of the Gramian  that allow for efficient approximation algorithms.

\subsection{Our results\label{introex}}

This note has two main results. The first concerns control energy at the first time the system can become controllable, namely  at time $$\tm := \lceil n/k \rceil - 1,$$ where $\lceil x \rceil$ is the smallest integer which is at least $x$ and $k$, recall, is the number of columns of $B$.  It is immediate that if $t<\tm$, then $W(t)$ is singular because  there are not enough columns for the controllability matrix to be full-rank.  Our first result shows that if the eigenvalues of $A$ lie in a set with logarithmic capacity (to be formally defined later) smaller than one, then  $W(\tm)$ has an eigenvalue upper bounded by something that decays to zero exponentially fast in $\tm$. 

Our second result considers $A$ which are Hermitian with $m$ eigenvalues which are stable (i.e, which lie in $[-1,1]$). Roughly speaking, we show that if $t = O \left( \left( \frac{m}{k} \right)^{2-\epsilon} \right)$ for\footnote{We use the standard $O$-notation, i.e., the statement $f=O(g)$ where $f$ and $g$ are positive quantities denotes the existence of a constant $C$ such that $f \leq C g$.} some $\epsilon>0$, the controllability Gramian $W(t)$ has an eigenvalue which is upper bounded by a quantity that goes to zero as $m/k \rightarrow +\infty$.  For example, if $t= O \left( \left( m/k \right)^{3/2} \right)$, our result gives the bound $\lambda_{\rm min} \left( W(t) \right) = O \left( (m/k)^{3/2} e^{-\sqrt{m/k}} \right)$ in this case.

The formal statements of these results are a little involved and will be given within the body of this paper.  We conclude our summary by illustrating their use on some simple examples.  Suppose $x(t+1) = A x(t) + b u(t)$ where $A \in \C^{n \times n}$ is diagonalizable as $V A V^{-1} = D$ and  $b \in \R^{n \times 1}$ is a vector of unit norm. Then:

\begin{itemize} \item If the eigenvalues of $A$ are contained within any equilateral triangle of side length $2$ in the complex plane, then there is some $n_0$ such that for all $n \geq n_0$, we have 
\begin{equation} \label{equilateral} \lambda_{\rm min} \left( W(n-1) \right) \leq ||V||_2^2 ||V^{-1}||_2^2 \cdot 0.133^n. \end{equation}
By contrast, if all the eigenvalues are contained within a circle in the complex plane of the very same area as this equilateral triangle, our methods give the bound
\begin{equation} \lambda_{\rm min} \left( W(n-1) \right) \leq ||V||_2^2 ||V^{-1}||_2^2 \cdot 0.552^n, \label{circle} \end{equation}
once again for $n$ large enough. 
\item Suppose $n=10,000$ and $A$ is a Hermitian matrix at least half of whose eigenvalues are stable. 
It turns out that this is enough information to conclude that \begin{small} \begin{eqnarray}  \lambda_{\rm min}(W(250,000)) & \leq&  1.03 \times 10^{-37} \label{first-sym} \\
\lambda_{\rm min}(W(1,000,000)) & \leq&  1.58 \times 10^{-4} \label{second-sym} 
\end{eqnarray} \end{small} In other words, the presence of many stable modes appears to be a significant obstacle to the efficient control of Hermitian systems.  
\end{itemize}

\subsection{Organization of this paper} We conclude the introduction with Section \ref{notat} which introduces some notation as well as some background facts which we will draw on throughout this paper. Section \ref{seclog} is dedicated to proving our first main result, namely the bound on $W(\tm)$ in terms of logarithmic capacity. Section \ref{secsymm} proves our second main result, which bounds the eigenvalues of $W(t)$ for a range of times $t$ in the special case when the matrix $A$ is Hermitian with many stable eigenvalues. We end with some concluding remarks in Section \ref{concl}.

\subsection{Notation and background\label{notat}}

We first describe some notation that we will use for the remainder of the paper. We use the standard notation 
$o_l(1)$ to denote any function of $l$ that goes to zero as $l \rightarrow +\infty$. Given a matrix $V$, its condition number is defined as ${\rm cond}(V) := ||V||_2 ||V^{-1}||_2$. The Frobenius norm of $V$ is denoted as $||V||_F$. As is standard, $V^*$ will denote the conjugate transpose of $V$ and $\overline{V}$ will denote its (elementwise) complex conjugate. We will use $\mathcal{P}_{j}$ to denote the set of univariate polynomials with complex coefficients of degree at most $j$, and $\mathcal{P}_j'$ to denote the set of {\em monic\footnote{Meaning the coefficient in front of the highest power is one.}} univariate polynomials with complex coefficients of degree $j$. 

Given a compact set $K$ in the complex plane, let $\mu^K$ be the set of probability measures on $K$, i.e., the set of Borel measures $\mu$ supported on $K$ which satisfy $\mu(K)=1$.  We then define $I(K)$, called the logarithmic energy of the
set $K$, as
\[ I(K) := \sup_{\mu \in \mu^K} ~\int_{K \times K} \log |z-w| ~ d\mu(z) d\mu(w). \]  The logarithmic capacity is then defined as 
\[ {\rm cap}(K) := e^{I(K)}. \] 

Logarithmic capacity comes up in our results due to its connection with polynomial approximation, which we now describe. Given a set $\X \subset \C$, we define
\begin{equation} \label{endef} {\rm Err}(l, \X) := \min_{p \in \mathcal{P}_{l-1}} \max_{z \in X} \left| z^l - p(z) \right|. \end{equation} In other words, ${\rm Err}(l,\X)$ is the best possible approximation error of the
function $z^l$ by a polynomial of degree $l-1$ over $\X$. The following statement is part of the ``fundamental theorem of potential theory'' from \citep{levinsaff}:
\[ \lim_{l \rightarrow \infty} \left( {\rm Err}(l,\X) \right)^{1/l} = {\rm cap}(\X).  \] As long as ${\rm cap}(\X) > 0$, we may re-arrange this as 
\begin{equation} \label{cap} {\rm Err}(l,\X) = \left[ \left( 1 + o_l(1) \right) {\rm cap}(\X) \right]^l. \end{equation} In other words, 
the logarithmic capacity of a region determines the growth (or decay) of the polynomial approximation of $z^l$ over that region by polynomials of 
lower degree.

%


\section{Control energy at time $\tm$ and logarithmic capacity\label{seclog}}

We begin with a statement of first our main result. Informally, the 
theorem says that if the eigenvalues of the matrix $A$ lie in the set $\X$, then $\lambda_{\rm min}(W(\tm))$ should scale roughly as $\sim \left( {\rm cap}(\X)^2 \right)^{\tm}$.  The formal statement is given next.

\begin{theorem} Suppose  that $A$ is diagonalizable matrix with $VAV^{-1} = D$ where $D$ is diagonal and with all of its eigenvalues within a set $\X \subset \C$ with $\cap > 0$.  We then have, 
\begin{align*} \lambda_{\rm min} \left( W \left( \tm \right) \right) &  \leq  {\rm cond}^2(V)  ~ ||B||_F^2 \\ & ~~~~\left( \left(1 + o_{\tm}(1) \right)  {\rm cap}^{2} (\X) \right)^{t_{\rm min} }    \end{align*} \label{mainthm}
\end{theorem}

We briefly sketch some intuition for this statement. Observe that if $n>k$ and the (diagonalizable) matrix $A$ only has one eigenvalue (necessarily of multiplicity $n$), then an easy implication of the Popov-Belevitch-Hautus theorem is that the system is uncontrollable. It is therefore reasonable to guess that this fact should have a quantitative extension, namely  some statement to the effect that if all the eigenvalues are close together then the controllability Grammian is close to singular. One formalization of this  is precisely the above theorem. 

Moreover, diagonalizability is crucial for any such bound; consider choosing $A$ to be the lower-shift matrix and $b=e_{1}$. Indeed, for this pair $A,b$, we have that the set $\{0\}$ contains the (single) eigenvalue of the system while the controllability Grammian $W(t)$ has all of its eigenvalues equal to $1$ whenever $t \geq \tm$. Thus for non-diagonalizable systems, having eigenvalues contained in a small set does not translate into good bounds on control energy. Furthermore, by considering diagonalizable approximations to this system, we see any bound we derive cannot depend solely on eigenvalues, and must somehow ``blow up'' whenever the system approaches this non-diagonalizable system, explaining the presence of the ${\rm cond}^2(V)$ term. 


Our first task is to prove the theorem; later in this section, we will show how to use this result to obtain more ``concrete'' estimates on $\lambda_{\rm min}(W(\tm))$. The idea of the proof is to argue that, in the definition of $W(\tm)$ in Eq. (\ref{gramian}), the final $k$ terms have, in an appropriately defined sense, small contributions and consequently $W(\tm)$ is close to singularity. This is a technique we have taken from the papers \citep{tyrt, zam} which we here combine with the arguments of \citep{pasq}.

\begin{proof}[Proof of Theorem \ref{mainthm}] From the definition of the controllability Gramian, we will then have
\begin{align*} V W(t) V^{*} & = \sum_{i=0}^{t} \left( V A^i V^{-1} \right) \left(V B \right) \\ &~~~~~~~~~~~~ \left(  B^* V^{*} \right) \left( V^{-*} (A^*)^i V^* \right). \end{align*}
Introducing the notation $Q(t) := V W(t) V^{*}$ and $Z := VB$, we can rewrite this as
\[ Q(t) = \sum_{i=0}^{t} D^i Z Z^* \overline{D}^i. \]  We will now derive an upper bound on the smallest eigenvalue of $Q(\tm)$, which we will later translate into a bound 
on the smallest eigenvalue of $W(\tm)$. To analyze the smallest eigenvalue of $Q(\tm)$, we introduce the following notation.

We define $z_1, \ldots, z_k$ be the columns of the matrix $Z$, and we define $S(t)$ to be the matrix \[ S(t) := [Z, ~DZ, \ldots, ~D^{t} Z  ]. \]  Observe that  $S(t) \in \C^{n \times (t+1)k}$ and has the columns $D^j z_i$ as $j$ runs over $j=0, \ldots, t$ and $i$ runs over $i=1, \ldots, k$. Moreover, 
\begin{equation} \label{sdef} S(t) S(t)^* = Q(t). \end{equation} Consequently, 
\[ \lambda_n(Q(t)) = \sigma_n^2(S(t)). \]
 Inspecting the last equation, we see that our goal of $\lambda_{\rm min}(Q(\tm))$ can be achieved by upper bounding the smallest singular value of $S(\tm)$. We will do so by arguing that $S(\tm)$ is close to singular. 

The argument proceeds as follows. We define the subspace
 $${\mathcal S}_i^j := {\rm span} \left(z_i, ~D z_i,  \ldots, ~D^{j} z_i \right).$$ We adopt  the  notation of $P_{\Y}$ to denote the matrix which projects onto the subspace $\Y$.  Now the projection of $D^{j+1} z_i$ onto the subspace ${\mathcal S}_i^j$ is naturally some linear combination of the
vectors which span ${\mathcal S}_i^j$, i.e., \begin{small} \begin{align*} P_{\mathcal{S}_i^j} D^{j+1} z_i  & = \alpha_j D^j z_i +    \cdots +  \alpha_1 D z_i + \alpha_0  z_i,
\end{align*} \end{small} for some (possibly complex) coefficients $\alpha_j, \ldots, \alpha_0$. Since $P_{\Y^\perp} (y) = y - P_{\Y}(y)$, we have that \begin{footnotesize}
\begin{align*} P_{(\mathcal{S}_i^j)^\perp} D^{j+1} z_i  & = D^{j+1} z_i - \alpha_j D^j z_i      \cdots -  \alpha_1 D z_i - \alpha_0  z_i,
\end{align*} \end{footnotesize} which we may write as
$P_{(S_i^j)^\perp} D^{j+1} z_i   =  p(D) z_i$, where \[ p(t) := t^{j+1} - \alpha_j t^j - \cdots - \alpha_1 t -  \alpha_0. \] Note that $p(t)$ is a monic polynomial of degree $j+1$, i.e., $p(t) \in \mathcal{P}_{j+1}'$. 

Now since the projection operator onto the subspace $\mathcal{X}$ maps every point to the closest point in $\mathcal{X}$, we have that 
\begin{equation} \label{minchar}  \left| \left| P_{(S_i^j)^\perp} D^{j+1} z_i  \right| \right|_2 = \min_{p \in \mathcal{P}_{j+1}'} \left| \left| p \left(D \right) z_i \right| \right|_2. \end{equation} 

\smallskip

Indeed, if Eq. (\ref{minchar}) did not hold, then there would be a point in ${\mathcal S}_i^j$ which is closer to $D^{j+1} z_i$ than
$P_{\mathcal{S}_i^j} (D^{j+1} z_i)$, which cannot be. 

Using the fact that entries of $D$ lie in the set $\X$ and recalling the definition of ${\rm Err}(j,\X)$ from Eq. (\ref{endef}), we have 
\begin{equation} \left| \left| P_{(S_i^j)^\perp} D^{j+1} z_i \right| \right|_2  \leq   {\rm Err}(j+1,\X) ||z_i||_2. \label{Eappears}
\end{equation}


Recall that our goal is to argue that $S(\tm)$  is close to a singular matrix, and to this end we use the last inequality as follows. We define $L(\tm)$ to be the matrix \begin{small}
\[ L(\tm) := [Z, ~~DZ, \ldots, ~~D^{\tm-1} Z, ~~P_{S_i^{\tm -1}} D^{\tm} Z] \] 
\end{small} In other words, the definition of  $L(\tm)$ is identical to $S(\tm)$ with the exception of the final $k$ columns, which are now multiplied by the
projection matrix $P_{S_i^{\tm -1}}$.  

Using Eq. (\ref{Eappears}), we can conclude that $L(\tm)$ and $S(\tm)$ are not too far apart: \begin{small}
\begin{align} ||S(\tm) - L(\tm)||_2^2 & \leq  ||S(\tm) - L(\tm)||_{\rm F}^2 \\ & \leq  \sum_{i=1}^k  \left| \left|  P_{\left( {\mathcal S}_i^{\tm - 1}\right)^\perp}  D^{\tm} z_i  \right| \right|_2^2 \nonumber  \\ 
& \leq  \left( {\rm Err}(\tm, \X)  \right)^2 \sum_{i=1}^k ||z_i||_2^2  \nonumber \\
& \leq   \left( {\rm Err}(\tm, \X)  \right)^2 ~||V||_2^2 ~ ||B||_F^2, \label{sbound} \end{align} \end{small} 
On the other hand,
\begin{equation} \label{rankbound}  {\rm rank}(L(\tm)) \leq k \tm < k \frac{n}{k} = n. 
\end{equation}  

Next we can use the standard interpretation of the singular values in terms of the distance to the 
closest low-rank matrix, i.e.,
\[ \sigma_n^2(S(\tm)) \leq \inf_{{\rm rank} ~M < n} ~||S(\tm) - M||_2^2, \] see e.g., Section 7.4.2 of \citep{hornj}. Putting this together with 
Eq. (\ref{rankbound}) and Eq. (\ref{sbound}), we obtain 
\begin{equation} \label{sigmabound} \sigma_n^2(S(\tm)) \leq \left( {\rm Err}(\tm, \X)  \right)^2 ~||V||_2^2 ~ ||B||_F^2.
\end{equation}

We now put together the various estimates we have derived. The first step is bound $\lambda_n(W(\tm))$ in terms of $\lambda_n(Q(\tm))$:
\begin{align*} \lambda_n(W(\tm)) &  = \lambda_n (V^{-1} Q(\tm) V^{-*}) \\[0.5ex] 
& =  \min_{x^* x = 1} ~x^* V^{-1} Q(\tm) V^{-*} x \\[0.5ex] 
& = \min_{y = V^{-*} x, x^*x=1} ~~ y^* Q(\tm) y 
\end{align*} Observe, however, that since $Q(\tm)$ is Hermitian, $y^* Q(\tm) y$ is nonnegative; furthermore, as $x$ ranges over the unit sphere in $\C^n$, $V^{-*} x$ ranges over 
a nondegenerate ellipsoid whose every element has norm at most $||V^{-*}||_2$. Therefore, 
\begin{align} \lambda_n(W(\tm))  & \leq \min_{||y||_2 = ||V^{-*}||_2} ~ y^* Q(\tm) y \nonumber \\[0.5ex]
& \leq  ||V^{-*}||_2^2  ~\lambda_{n} \left( Q(\tm) \right) \nonumber \nonumber \\[0.5ex] 
& =   ||V^{-1}||_2^2  ~\sigma_n^2(S(\tm))  \nonumber \nonumber \\[0.5ex] 
& =   ||V^{-1}||_2^2 ~ ||V||_2^2  ~\left( {\rm Err}(\tm, \X)  \right)^2 ||B||_F^2 \nonumber  \\[0.5ex] 
& = {\rm cond}^2(V) ~\left( {\rm Err}(\tm, \X)  \right)^2 ||B||_F^2 \label{approxform} 
\end{align}  Combining the last equation with Eq. (\ref{cap}) (which bounds  $ {\rm Err}(\tm, \X) $ in terms of $\cap$) completes the proof. 
\end{proof}

%

\subsection{Some concrete estimates\label{concrete}}

The logarithmic capacities of many different sets have been worked out in the literature, and we can use them to spell out some 
concrete forms of Theorem \ref{mainthm}. 

Let us begin by revisiting the example we have previously mentioned in Section \ref{introex}. Suppose that $A \in \C^{n \times n}$ has all of its eigenvalues in an equilateral triangle $\X$ of length $l$; further, suppose $b={\bf e}_1$. Then it is known that  \citep{dijkstra}, $$\cap = \frac{(\Gamma(1/3))^2}{4 \pi^2} l \approx 0.18\cdot  l,$$ so that 
Theorem \ref{mainthm} specializes to the bound 
\[  \lambda_{\rm min} \left( W \left( n-1 \right) \right)   \leq  {\rm cond}^2(V)     \left( \left(1 + o_{n}(1) \right) \cdot 0.18^2 \cdot l^2 \right)^{ n} 
\] Plugging in $l=2$ and using the fact that $o_{n}(1)$ becomes arbitrarily small for large enough $n$, we obtain Eq. (\ref{equilateral}) from Section \ref{introex}. We may view this last equation as a somewhat more concrete form of Theorem \ref{mainthm}. 

Along the same lines, the following proposition collects a number of estimates of the logarithmic capacity from the literature. 

\begin{proposition}  Suppose that $\X$ is:
\begin{itemize} 
\item ...ellipse with semi-axes $a$ and $b$. Then \citep{dijkstra}, $\cap = (a+b)/2$. 
\item ...a rectifiable curve length $l$. Then \citep{dubinin}, $\cap \leq l/4$. In particular, if $X = [a,b]$ then \citep{polya2}, $\cap = (b-a)/4$.
\item ...two intervals on the real line of the form $[-b,-a] \cup [a,b]$. Then \citep{liesen}, $$\cap = \frac{\sqrt{b^2-a^2}}{2}.$$
\item ...a half disk of radius $r$. Then \citep{land}, $\cap = 4r/3^{3/2}$. 
\item ...square with side $l$. Then \citep{dijkstra},  $$\cap = \frac{(\Gamma(1/4))^2}{4\pi^2} l \approx 0.59 l.$$ 
\item ...a regular $n$-gon with side $h$. Then \citep{land},
$$ \cap = \frac{\Gamma(1/n)}{2^{1+2/n} \pi^{1/2} \Gamma(1/2 + 1/n)} h. $$
\item ...a compact set of diameter $D$. Then (\citep{polya}, quoted in \citep{barnard}) $\cap \leq D/2$. 

\end{itemize}
\end{proposition} 

One may use this proposition to come up with statements similar to Eq. (\ref{equilateral}) for other sets of the complex plane. For example, one can use the
first bullet of the proposition on the logarithmic capacity of the circle to obtain Eq. (\ref{circle}) from Section \ref{introex}; we omit the details of the calculation.

\section{A lower bound on control energy for Hermitian systems with stable eigenvalues\label{secsymm}} 

In this section we will consider the case of Hermitian matrices $A$ with $m$ eigenvalues within the interval $[-1,1]$. We will be able to obtain
stronger results in this case, in particular showing that $W(t)$ has an eigenvalue close to zero until $t$ is nearly quadratic in $m/k$. A formal statement
of this is in the following theorem. 

\begin{theorem} Suppose $A$ is a Hermitian matrix with $m$ stable eigenvalues where $m > 2k$. Let $q$ be any nonnegative number and set $$t_{\rm quad} :=   \frac{(\lceil m/k \rceil - 2)^2}{q}. $$ Then if $t \leq  t_{\rm quad} $  we have
\[ \lambda_n(W(t)) \leq 4 \frac{\left( \lceil m/k \rceil - 2 \right)^2}{q} e^{-q} ||B||_F^2. \] \label{mainsymm}
\end{theorem}

We remark that by plugging in $m=5000$, $k=1$, $||B||_F=1$, and $q=99$, we can obtain Eq. (\ref{first-sym}), which was discussed in the introduction. The next Eq. (\ref{second-sym}) can be obtained by plugging in $q=24$ with the same
values of $m$, $k$, and $||B||_F$. We also remark that the theorem holds for any nonnegative $q$ and one can, for example, plug in $q = \sqrt{m/k}$ to obtain the bound
$\lambda_{\rm min} \left(W (t) \right) \leq O \left( (m/k)^{3/2} e^{-\sqrt{m/k}} \right)$ when $t = O \left( (m/k)^{3/2} \right)$ (we mentioned this inequality
 in the introduction).

To prove Theorem \ref{mainsymm}, we will need to derive some additional properties of polynomial approximations over the interval $[-1,1]$. For integers $n,m$ satisfying $n \geq m \geq 1$, we define 
\[ \Phi_{n,m} :=  \min_{p_m \in { \Pi}_m} \max_{|x| \leq 1} \left| x^n - p_m(x) \right|, \] where $\Pi_m$ is the set of polynomials of degree $m$ with real coefficients. In other words, $\Phi_{n,m}$ is the  error involved in approximating the function $x^n$ over $[-1,1]$ by an $m$'th degree polynomial in $x$.

It has been observed in \citep{zam} that when $n \geq m \geq cn$ for some fixed $c$, the quantity $\Phi_{n,m}$ decays exponentially in $n$. The following lemma gives an explicit estimate of the decay rate.

\begin{lemma} \[  \Phi_{n,m} \leq 2 e^{-m^2/(2n)} \]  \label{philemma}
\end{lemma} 

\begin{proof} Just like the corresponding proof in \citep{zam}, our starting point is the well-known identity
\begin{equation} \label{x-exp} x^n = \frac{1}{2^{n-1}} \sum_{i=0}^{\lfloor n/2 \rfloor} {n \choose i} T_{n - 2i}(x) \frac{\delta_{i,n}}{2} \end{equation} where $T_k(x)$ the  Chebyshev polynomial of degree $k$, and $\delta_{i,n}$ equals $1$ if $i=n/2$ and $2$ otherwise (for a reference, see Eq. (2.14) in \citep{chebbook}).  
Let $i'$ be the first integer such that $n - 2i' \leq m$, i.e., $i' = \lceil (n-m)/2 \rceil$. Then, to approximate $x^n$ by a polynomial of degree at most $m$, we might try the polynomial 
\[ p_{n,m} = \frac{1}{2^{n-1}} \sum_{i=i'}^{\lfloor n/2 \rfloor} {n \choose i} T_{n - 2i} \frac{\delta_{i,n}}{2}, \] which, note, has degree at most $m$.  As a consequence of Eq. (\ref{x-exp}) as well as the fact that $|T_k(x)| \leq 1$ for all $x \in [-1,1]$, this leads to the bound \begin{small}
\begin{eqnarray*} \Phi_{n,m}   \leq   \frac{1}{2^{n-1}} \sum_{i =0, \ldots, i'-1 } {n \choose i} 
 =  2 \left(  \frac{1}{2^{n}} \sum_{i =0}^{i'-1} {n \choose i} \right)
\end{eqnarray*} \end{small} The expression in big parentheses is the probability that a fair coin lands on heads at most $i' - 1$ times 
out of $n$ tosses. We use the following form of Hoeffding's inequality 
\[ \frac{1}{2^{n}} \sum_{i =0, \ldots, (1/2-\epsilon) n} {n \choose i}  \leq e^{-2 \epsilon^2 n}, \] for a formal reference see Section B.4 in \citep{ml}. Since $i' -1 \leq n/2$, we may apply this form of the Hoeffding bound to obtain 
\begin{align*} \Phi_{n,m} & \leq 2 e^{-2n (1/2 - ( i' - 1 )/n)^2 } \\
& = 2 e^{-2n (1/2 -  ( \lceil (n-m)/2 \rceil - 1)/(n)   )^2} \\
& \leq 2 e^{-2n(\frac{1}{2} - \frac{(n-m)/2}{n})^2} \\
& \leq  2 e^{-m^2/(2n)}
\end{align*}
\end{proof} 



%

\begin{lemma} Let $q$ be any nonnegative number. We have that 
if $m \geq 1$,
$$\sum_{n=m}^{\lfloor m^2/q \rfloor} \Phi_{n,m}^2 \leq 4 \frac{m^2}{q} e^{- q}. $$ \label{sumlemma}
\end{lemma}

\begin{proof} Indeed, by Lemma \ref{philemma}, 
\begin{align*} \sum_{n=m}^{\lfloor m^2/q \rfloor} \Phi_{n,m}^2  & \leq   \sum_{n=m}^{\lfloor m^2/q \rfloor} \left( 2 e^{-m^2/(2n)} \right)^2  \\[1.5ex]
& \leq 4   \sum_{n=m}^{\lfloor m^2/q \rfloor}  e^{-m^2/n}  
\end{align*}
The number of terms in the last sum is upper bounded by $m^2/q$, and each term is upper bounded by $$e^{-\frac{m^2}{m^2/q}} = e^{-q}.$$ This immediately implies
the lemma. 
\end{proof}

With this lemma in hand, we are now ready to prove Theorem \ref{mainsymm}.

\begin{proof}[Proof of Theorem \ref{mainsymm}] As in the proof of Theorem \ref{mainthm}, let $$V A V^{-1} = D$$ be the diagonalization of $A$; here $V$ is unitary since $A$ is Hermitian, and we will assume without loss of generality that $d_{11}, d_{22}, \ldots, d_{mm} \in [-1,1]$.

 Let $Q(t) = V W(t) V^*$ be the same as defined earlier.  As a consequence of the fact that $V$ is unitary, we have that $$\lambda_n(W(t)) = \lambda_n(Q(t)).$$

We define $D(m)$ to be the principal $m \times m$ submatrix\footnote{That is, the submatrix obtained by taking the first $m$ rows and first $m$ columns.} of the diagonal matrix $D$, let $z_i(m)$ be the vector obtained by taking the first $m$ entries of the vector\footnote{Recall that $Z=VB$ and $z_1, \ldots, z_k$ are the columns of the matrix $Z$.} $z_i$, 
and let $S(m,t)$ be the matrix with columns $D(m)^j z_i(m)$ for $i = 1, \ldots, k$ and $j=0, \ldots, t$. As before, $S(t)$ is the matrix with columns $D^j z_i$ with $i,j$ running
over the same ranges; recall that $Q(t) = S(t) S(t)^*$. Finally, abusing notation slightly, we will now use $S_i^j$ to denote the span of the vectors $z_i(m), D(m) z_i(m), \ldots, D(m)^j z_i(m)$. 

Using the Courant-Fischer theorem, we obtain \begin{small}
\begin{align} \lambda_{\rm min} (Q(t)) & =  \min_{x \in \C^n, ~x^*x=1} ~x^* S(t) S(t)^* x \nonumber \\[1.5ex]
& \leq  \min_{\substack{x \in \C^n, ~x^* x = 1 \\ x_{m+1}=x_{m+2} = \cdots = x_n = 0}} ~~~~ x^* S(t) S(t)^* x \nonumber \\[1.5ex]
& =  \min_{x \in \C^m, ~x^* x = 1} ~~ x^* S(m, t) S(m, t)^* x \nonumber \\[1.5ex]
& =  \sigma_{m}^2(S(m,t)). \label{sigmalambda} 
\end{align}  \end{small} We thus have 
\begin{equation} \label{lams} \lambda_n(W(t)) \leq \sigma_m^2(S(m,t)). \end{equation}
We now turn to the problem of upper bounding the $m$'th smallest singular value of $S(m,t)$, which we will do by arguing, along the lines of the proof of Theorem \ref{mainthm}, that $S(m,t)$ is well-approximated by a matrix whose rank is strictly less than $m$. 

Indeed, define $t'=\lceil m/k \rceil - 1$. We may as well assume  that $t \geq t'$ since otherwise there is nothing to prove. Define further the matrix $\widehat{L}(m,t)$ as having its first $k t'$ columns being $D(m)^j z_i(m)$, $i=1, \ldots, k$ and $j = 0, \ldots, t' -1$; however,
its final $k(t+1-t')$ columns will be the
vectors $P_{{\mathcal S}_i^{t' - 1}} D(m)^{j} z_i(m)$, $i=1, \ldots, k$ and $j=t', \ldots, t$. Note that $\widehat{L}(m,t)$ is roughly analogous to the matrix $L(m,\tm)$ defined in the proof of Theorem \ref{mainthm}. 

As in the proof of Theorem \ref{mainthm}, we have that
\[ {\rm rank} (\widehat{L}(m,t)) \leq k t' < k \frac{m}{k} = m. \]  We will next argue that $\widehat{L}(m,t)$ and $S(m,t)$ are close to each other. Indeed, we have that for all $l \geq j$, \begin{footnotesize}
\begin{align*} \left| \left| P_{(S_i^j)^\perp} D(m)^{l} z_i(m) \right| \right|_2  & = \min_{p \in {\cal P}_j} \left| \left| \left( D(m)^l - p(D(m)) \right) z_i(m) \right| \right|_2 \\
& \leq   \Phi_{k,j} ||z_i(m)||_2,
\end{align*}
\end{footnotesize} where the first step follows from definition of projection and the second step used the fact that all diagonal entries of $D(m)$ belong to $[-1,1]$. Thus, 
\begin{footnotesize}
\begin{align*} ||S(m,t) - \widehat{L}(m, t)||_2^2 & \leq  ||S(m,t) - \widehat{L}(m, t)||_{\rm F}^2 \\ & \leq  \sum_{i=1}^k \sum_{l=t'}^{t} \left| \left|  P_{\left({\mathcal S}_i^{t' - 1}\right)^\perp}  D(m)^{l} z_i(m)  \right| \right|_2^2 \nonumber  \\
& \leq \sum_{i=1}^k  ||z_i(m)||_2^2 \sum_{l=t'}^{\lfloor t_{\rm quad} \rfloor} \Phi_{l, t'-1}^2 
\end{align*}  \end{footnotesize} where we used the fact that $t$ is an integer and $t \leq t_{\rm quad}$ in the last inequality.  Note that we may as well begin the second sum above at $l=t'-1$. Observe that by
definition $t_{\rm quad} = (t'-1)^2/q$, and thus we can use Lemma \ref{sumlemma} to bound the sum. Furthermore, due to the assumption $m>2k$ we have $t'-1 \geq 1$.  Thus applying Lemma \ref{sumlemma}, we obtain \begin{small}
\[ ||S(m,t) - \widehat{L}(m, t)||_2^2 \leq 4 \frac{(\lceil m/k \rceil - 2)^2}{q } e^{-q} ||B||_F^2. \] 
\end{small}
Putting it all together: this inequality along with ${\rm rank} (\widehat{L}(m,t)) < m$   implies that the right-hand side is an upper bound on $\sigma_m^2(S(m,t))$,  and via Eq. (\ref{lams}) it is also an upper bound on $\lambda_n(W(t))$.

\end{proof}


\section{Conclusion\label{concl}} We have proven two bounds to the effect that if the eigenvalues of the matrix $A$ are clustered, the linear system of Eq. (\ref{maineq}) is difficult to control. Our work points the way to a number of open questions. 

First, it is unclear whether it is possible to extend the bounds of Theorem \ref{mainthm} for $t$ that are much greater than $\tm$. The antecedent work \citep{pasq} bounded the eigenvalues of $W(t)$ for all $t$ under the assumption of clustering within the interior of the unit circle, and it is at present  unclear if Theorem \ref{mainthm} can be extended in such a way. Second, we conjecture that the bounds of Theorem \ref{mainsymm} are essentially tight; specifically, we conjecture that for all $n$ there exists a symmetric matrix in $\R^{n \times n}$ with all of its eigenvalues in $[-1,1]$ and with $\lambda_{\rm min}(W(t)) \geq 1$ for all $t \geq \Omega(n^2 \log^k n)$ for some $k$. 

More broadly, the main contribution of this note is to highlight the apparently subtle connection between eigenvalue clustering and control energy (recall the contrast between the bounds we discussed in Section \ref{introex}). An open problem is to understand this connection better.


\end{document}